\documentclass[twoside, 12pt]{article}

\usepackage[T2A]{fontenc}
\usepackage[utf8]{inputenc}
\usepackage[english]{babel}

\usepackage[width=155mm, left=20mm, top=20mm, height=230mm, paper=a4paper]{geometry}

\usepackage[usenames,dvipsnames,svgnames,table]{xcolor}

\usepackage{amsfonts,amsmath,amssymb,amsthm}

\usepackage{mathrsfs}

\usepackage{scalerel}

\usepackage{color}
\definecolor{darkred}{rgb}{0.6,0,0.1}
\definecolor{darkgreen}{rgb}{0.2,0.7,0.4}%{0.1,0.4,0.4}
\definecolor{darkblue}{rgb}{0.1,0.4375,0.8125}
\definecolor{lightblue}{rgb}{0.16015625,0.4453125,0.6328125}

\usepackage[colorlinks,unicode,citecolor={darkgreen}, linkcolor=blue,hyperindex=true]{hyperref}

 \def\@evenhead{\vbox{\hbox to \textwidth{\thepage\hfil\sl\leftmark\strut}\hrule}}
 \def\@oddhead{\vbox{\hbox to \textwidth{\rightmark\hfill\thepage\strut}\hrule}}

\def\cupt{\mathop{\textstyle\bigcup}}
\def\capt{\mathop{\textstyle\bigcap}}

\def\Sshv{\mathscr{S}(\Re^n)}
\def\Sshvf{\mathscr{S}'(\Re^n)}

\def\Cs#1{C_{#1}(\Re^n)}
\def\CDs#1#2{C_{#1}^{#2}(\Re^n)}

\def\Re{\mathbb R}
\def\Na{\mathbb N}
\def\Ze{\mathbb Z}
\def\Co{\mathbb C}

\def\izmer{\mathfrak{M}^n}

\def\Hasp{H^1(\Re^n)}
\def\Haspa#1{H^1_{atom}(\Re^n)}

\def\lOf#1{\Oplot{[#1]}}
\def\Maxfv{{\cal M}}
\def\BMO#1{\mathrm{BMO}(\Re^{#1})}
\def\VMO#1{\mathrm{VMO}(\Re^{#1})}

\def\Avg#1#2{\mathop{\rm Avg}_{#2}\left(#1\right)}

\def\funa{u}
\def\funb{\Upsilon}
\def\funf#1{\Opfunk_{#1}}
\def\Oplot#1{\ell_{#1}}
\def\Opfunk{J}

\def\Real{\mathrm{Re\,}}
\def\Imag{\mathrm{Im\,}}

\def\strong{\mathrm{s}}
\def\weak{\mathrm{w}}
\def\lebd#1{{\mathcal{L}^{#1}}}
\def\loc{\text{\rm loc}}

\def\esup{\mathop{\rm ess\,sup}}

\newcounter{Noi}
\newenvironment{listi}%
{\begin{list}{\hskip16pt\llap{{\rm (\roman{Noi})}}\hskip4pt}{\usecounter{Noi}%
          \setlength{\labelsep}{0pt}\setlength{\leftmargin}{20pt}%
          \setlength{\itemsep}{0pt}\setlength{\labelwidth}{0pt}%
           \setlength{\parsep}{0pt}\setlength{\listparindent}{20pt}}}{\end{list}}

\def\supp{\mathop{\rm supp}\nolimits}

\def\Lsp#1{L_{#1}(\Re^n)}
\def\Lsploc#1{L_{#1}^\loc(\Re^n)}

\newcounter{punktcount}[section]
\renewcommand{\thepunktcount}{\thesection.\arabic{punktcount}}
\def\punkt#1{\medskip\par\noindent\refstepcounter{punktcount}{\bf \thepunktcount.~#1}}

\newcounter{stepofproof}[punktcount]

\newcounter{vspcount}
\renewcommand{\thevspcount}{\roman{vspcount}}

\def\retrompage#1{\setcounter{vspcount}{#1}\thevspcount}
\def\rom#1{{\rm (\retrompage{#1})}}

\newenvironment{thm}{\punkt{Theorem.}\it}{}
\newenvironment{lemma}{\punkt{Lemma.}\it}{}
\newenvironment{corollary}{\punkt{Corollary.}\it}{}

\begin{document}

\begin{center}{\bf ON THE ASSOCIATED SPACES  OF THE  HARDY  SPACE}
 \end{center}
     % Title of talk

\vskip 0.3cm

\centerline{\bf D.V.~Prokhorov}        % Authors from the same institution

\markboth{\hfill{\footnotesize\rm   D.V.~Prokhorov  }\hfill}
{\hfill{\footnotesize\sl  On the associated spaces of the  Hardy  space}\hfill}
\vskip 0.3cm

\vskip 0.7 cm


\vskip 0.2cm


\vskip 0.2cm

\noindent {\bf Abstract.} Characterizations  of the  associated spaces and second associated spaces of the Hardy space on $\Re^n$ are given. Some results on the associated spaces of the $\BMO{n}$ space are proved also.

\section{\large Introduction}

 Let $\Re^n$ be the $n$-dimensional Euclidean space and $
 \Re^n_{+}:=\{(x_1,\ldots,x_n)\in\Re^n: x_n>0\}$. By cube we mean a closed cube in $\Re^n$ with edges parallel to the coordinate axes; the symbol $I(x,l)$ denotes a cube with center at $x$ and edge length $l$.
 
 By $\lebd{n}$ we denote the $n$-dimensional Lebesgue measure on $\Re^n$,  $\izmer$  is the vector space of all $\lebd{n}$-measurable functions $f:\Re^n\to\Co$.
 For $p\in[1,\infty]$ we put $p':=\frac{p}{p-1}$. The Lebesgue spaces $\Lsp{p}$ and $\Lsploc{p}$ are defined as usual 
 \begin{align*}
  &\Lsp{p}:=\left\{f\in\izmer\,\Big|\,\|f\|_{\Lsp{p}}<\infty\right\},\\  &\Lsploc{p}:=\left\{f\in\izmer\,\Big|\,\|f\chi_K\|_{\Lsp{p}}<\infty \ \forall\text{ compact }K\subset\Re^n\right\},
 \end{align*}
 where 
\begin{equation*}
  \|f\|_{\Lsp{p}}:=\begin{cases}
           \left(\int_{\Re^n}|f|\,d\lebd{n}\right)^\frac{1}{p}, & p\in [1,\infty),\\
           \lebd{n}\text{-}\esup_{x\in\Re^n}|f(x)|, & p=\infty.
           \end{cases}
 \end{equation*}

$\Sshv$ is  the  space of  rapidly decreasing functions $f:\Re^n\to\Co$, and  $\Sshvf$  is its conjugate space.
If $\funa\in \Sshvf$ and there exists $f\in\izmer$ such that $\funa$ has the representation   
\begin{equation}\label{funkfun}
\phi\mapsto \int_{\Re^n}f\phi\,d\lebd{n},\ \ \phi\in \Sshv, 
\end{equation}
then the function $f$ we denote by $\Oplot{\funa}$ and call ``density function'' of the functional $\funa$. 
If $Y\subset\Sshvf$ such that there exists $\Oplot\funa$ for each $\funa\in Y$, then we define $\lOf{Y}:=\{\Oplot\funa:\funa\in Y\}$.
If $f\in\izmer$ such that  \eqref{funkfun} gives a functional from $\Sshvf$, then the functional \eqref{funkfun} we denote by $\funf{f}$. 

Fix a function $\varphi\in \Sshv$ with $\kappa_\varphi:=\int_{\Re^n} \varphi\,d\lebd{n}\not=0$. 
 The  Hardy space on $\Re^n$ is the space
\begin{equation*}\Hasp:=\Big\{\funa\in \Sshvf\,\Big|\,\exists \Oplot{\funa}\in\Lsp{1},\,\Maxfv_\varphi \Oplot{\funa}\in \Lsp{1}\Big\},
\end{equation*}
where
\begin{equation*}(\Maxfv_\varphi f)(x):=\sup_{t>0}\frac{1}{(2\pi)^\frac{n}{2}t^n}\left|\int_{\Re^n}\varphi\left(\frac{x-y}{t}\right)f(y)\,dy\right|,\ \ f\in\Lsp{1},\ x\in\Re^n.
\end{equation*}
The norm on $\Hasp$  is defined by the equality  $\|\funa\|_{\Hasp}:=\|\Maxfv_\varphi \Oplot{\funa}\|_{\Lsp{1}}$.  
We also put $\|f\|_{\lOf{\Hasp}}:=\|\funf{f}\|_{\Hasp}$ for $f\in\lOf{\Hasp}$. The definition of the space $\Hasp$ does not depend on the choice of function  $\varphi$, and  the norms $\|\cdot\|_{\Hasp}$ are equivalent. The properties of the space $\Hasp$ are described in detail in monograph \cite[Chapter III, IV]{SteinHA}.
The properties that we will use are given in Theorems \ref{lodin}--\ref{sopryazh}.
  
\begin{thm} \label{lodin} %Let $\varphi\in\Sshv$ and $\kappa_\varphi\not=0$. 
\begin{listi}
 \item \label{ocenkalh}  If $\funa\in\Hasp$ then $\int_{\Re^n} \Oplot{\funa}\,d\lebd{n}=0$ and $\|\Oplot{\funa}\|_{\Lsp{1}}\le \frac{(2\pi)^{\frac{n}{2}}}{|\kappa_\varphi|}\|\funa\|_{\Hasp}$.
 \item \label{lphsone} Let $p\in (1,\infty]$, $f\in \Lsp{p}$, $\supp f\subset Q$ for some cube $Q$, $\int_{\Re^n} f\,d\lebd{n}=0$. Then $\funf{f}\in\Hasp$ and $\|\funf{f}\|_{\Hasp}\le c^*_1(n,p,\varphi)\lebd{n}(Q)^\frac{1}{p'}\|f\|_{\Lsp{p}}$.
\end{listi}
 \end{thm}

An important role in the theory of the  Hardy spaces is played the atomic decomposition  of an element of the space.  A function $a:\Re^n\to \Co$ is an atom (associated to a cube $Q$) if \rom{1} $a$ is supported in $Q$, \rom{2} $\|a\|_{\Lsp{\infty}}\le 1$, and \rom{3} $\int_{\Re^n} a\,d\lebd{n}=0$. Note that $\funf{a}\in \Hasp$ for any atom $a$ by \ref{lphsone}. And we will call the element $\funf{a}$ an atom too. The symbol $\Haspa{q}$ denotes the subspace of $\Hasp$ consisting of all finite linear combinations of atoms of $\funf{a}\in \Hasp$.

 \begin{thm}\label{thrazlatoms}
  Let %$\varphi\in\Sshv$, $\kappa_\varphi\not=0$ and 
  $\funa\in \Sshvf$. 
  \begin{listi}
   \item \label{LambdainH} If there exist a sequence $\{b_i\}_1^\infty$ of   atoms and a sequence  $\{\lambda_i\}_1^\infty\subset \Co$ such that  $\sum_{i=1}^\infty|\lambda_i|<\infty$ and
  $\sum_{i=1}^j\lambda_i \funf{b_i}\to \funa$ as $j\to\infty$
  in weak\,$^*$ topology of $\Sshvf$, then $\funa\in\Hasp$,  
  \begin{equation}\label{predelvHodin}
\lim_{j\to\infty}\left\|\sum_{i=1}^j\lambda_i \funf{b_i}-\funa\right\|_{\Hasp}=0   
  \end{equation}
  and
  \begin{equation}
   \|\funa\|_{\Hasp}\le c^*_1(n,\infty,\varphi)\sum_{i=1}^\infty |\lambda_i|.
  \end{equation}
   \item \label{sushrazlozh} If $\funa\in\Hasp$, then there exist a constant $c^*_2(n,\varphi)>0$, a sequence $\{b_i\}_1^\infty$ of atoms and a sequence $\{\lambda_i\}_1^\infty\subset \Co$ such that  \eqref{predelvHodin} holds and
  \begin{equation}
   \sum_{i=1}^\infty |\lambda_i|\le c^*_2(n,\varphi)\|\funa\|_{\Hasp}.
  \end{equation}
  \end{listi}
 \end{thm}

 The dual space for $\Hasp$ is described in terms of elements of the $\BMO{n}$ space, defined below.
  
 For $f\in\Lsploc{1}$ and  $\lebd{n}$-measurable $E\subset\Re^n$ we put
  $\Avg{f}{E}:=\frac{1}{\lebd{n}(E)}\int_E f\,d\lebd{n}$.
By definition
\begin{equation*}
 \|f\|_{\BMO{n}}:=\sup_{Q}\frac{1}{\lebd{n}(Q)}\int_Q|f-\Avg{f}{Q}|\,d\lebd{n},\ \ f\in\Lsploc{1},
 \end{equation*}
 where the supremum is taken over all cubes $Q$, and 
 \begin{equation*}
 \BMO{n}:=\Big\{f\in\Lsploc{1}\,\Big|\,\|f\|_{\BMO{n}}<\infty\Big\}.
 \end{equation*}
The symbol $\VMO{n}$ denotes
the closure in the $\BMO{n}$ norm of the space $\Cs{c}$ of continuous functions with
compact support.

\begin{thm} \label{sopryazh} %Let $\varphi\in\Sshv$ and $\kappa_\varphi\not=0$.
\begin{listi}
 \item \label{pofbmofunk} For  $f\in\BMO{n}$ there exists unique $\funb\in(\Hasp)^*$ such that
 \begin{equation}\label{funknaatome}
\funb(\funf{a})=\int_{\Re^n} fa\,d\lebd{n}  
 \end{equation}
 holds for any atom $\funf{a}\in \Hasp$.
 And also, $\|\funb\|_{(\Hasp)^*}\le c^*_2(n,\varphi)\|f\|_{\BMO{n}}$.
\item  \label{predstfunkhardysp}
For  $\funb\in(\Hasp)^*$ there exists $f\in\BMO{n}$ such that \eqref{funknaatome} holds for any atom $\funf{a}\in \Hasp$. And also, $\|f\|_{\BMO{n}}\le 4c^*_1(n,2,\varphi) \|\funb\|_{(\Hasp)^*}$.
\end{listi}
\end{thm}

\section{\large Characterization of the associated spaces}

An accurate theory of associated (K\"othe dual) spaces of the Banach Function Spaces  can be found in book \cite[Chapter 1]{BenSharp}. For non-ideal spaces, two types of associated spaces can be considered \cite{Pr2022-1,PSU-UMN}.

Let $X$ be the vector subspace of $\izmer$ and a topology  on $X$ be define with help of a  seminorm $p_X:X\to [0,\infty)$. We define the ``strong'' associated space by 
\begin{equation*}
 X'_\strong:=(X,p_X)'_\strong:=\left\{g\in \izmer\,\Big|\,\exists C_\strong(g)>0:\int_{\Re^n}|hg|\,d\lebd{n}\le C_\strong(g)\,p_X(h)\ \forall h\in X\right\}
\end{equation*}
and the ``weak'' associated space 
\begin{align*}
X'_\weak:=(X,p_X)'_\weak:=&\left\{g\in \izmer\,\Big|\,fg\in \Lsp{1}\,\forall f\in X\vphantom{\int_{\Re^n}}\right.\\
&\left.\&\ \exists C_\weak(g)>0:\left|\int_{\Re^n}hg\,d\lebd{n}\right|\le C_\weak(g)\,p_X(h)\ \forall h\in X\right\}, 
\end{align*}
which is  isomorphic to the subspace of the set  $X^*$ of all continuous functionals of the form $f\mapsto \int_{\Re^n} fg\,d\lebd{n}$, $f\in X$. Clear $X'_\strong\subset X'_\weak$. Also we put  $\|g\|_{X'_\strong}:=\inf C_\strong(g)$ for $g\in X'_\strong$ and $\|g\|_{X'_\weak}:=\inf C_\weak(g)$ for $g\in X'_\weak$.

%\subsection{Associated spaces of the  Hardy space}

Since each element of $\Hasp$ has density function it is possible to consider of associated spaces of $\Hasp$, namely we have 
\begin{align*}&(\Hasp)'_\strong:=\Bigg\{g\in\izmer\,\Bigg|\,\|g\|_{(\Hasp)'_\strong}:=\sup_{\funa\in\Hasp\setminus\{0\}}\frac{\int_{\Re^n}|g\Oplot\funa|\,d\lebd{n}}{\|\funa\|_{\Hasp}}<\infty\Bigg\},\\
&(\Hasp)'_\weak:=\Bigg\{g\in\izmer\,\Bigg|\,\int_{\Re^n}|g\Oplot\funa|\,d\lebd{n}<\infty\ \forall\,\funa\in\Hasp\ \&\\
&\phantom{\hskip 40mm}\|g\|_{(\Hasp)'_\weak}:=\sup_{\funa\in\Hasp\setminus\{0\}}\frac{\left|\int_{\Re^n}g\Oplot\funa\,d\lebd{n}\right|}{\|\funa\|_{\Hasp}}<\infty\Bigg\}.
\end{align*}
  
For ``strong'' associated space of $\Hasp$  we have the following result.
  
\begin{thm} \label{silnohardiacco} %Let $\varphi\in\Sshv$ and $\kappa_\varphi\not=0$. Then 
 There is equivalence $g\in(\Hasp)'_\strong$ $\Leftrightarrow$ $g\in\Lsp{\infty}$. Moreover,
 \begin{equation*}
  \|g\|_{(\Hasp)'_\strong}\le \frac{(2\pi)^{\frac{n}{2}}}{|\kappa_\varphi|}\|g\|_{\Lsp{\infty}},\ \ \ 
\ \|g\|_{\Lsp{\infty}}\le c^*_1(n,\infty,\varphi)\|g\|_{(\Hasp)'_\strong}.
 \end{equation*}
\end{thm}
  
  \begin{proof}
   {\it Necessity.}  Let $g\in(\Hasp)'_\strong$.
   For an arbitrary cube $Q=I(x,l)$  we put
   \begin{equation}\label{delimkub}
   E_Q:=[(Q-x)\capt\Re^n_{+}]+x,\ \ a_Q:=\frac{1}{\lebd{n}(Q)}(\chi_{Q\setminus E_Q}-\chi_{E_Q}).
   \end{equation}
  Then $a_Q$ is an atom  associated to a cube $Q$. 
    Since $\funf{a_Q}\in\Hasp$ then we have $\int_{\Re^n}|ga_Q|\,d\lebd{n}<\infty$. Hence, $g\in\Lsploc{1}$.
   
 Let $x_0$ be a Lebesgue point of the function $|g|$. Then
 \begin{align*}
 |g(x_0)|&=\lim_{l\to 0+}\frac{1}{\lebd{n}(I(x_0,l))}\int_{I(x_0,l)}|g|\,d\lebd{n}=\lim_{l\to 0+}\int_{\Re^n}|ga_{I(x_0,l)}|\,d\lebd{n}\\
 &\le c^*_1(n,\infty,\varphi)\,\|g\|_{(\Hasp)'_\strong}. 
 \end{align*}
 Thus, $g\in\Lsp{\infty}$ and  $\|g\|_{\Lsp{\infty}}\le c^*_1(n,\infty,\varphi)\|g\|_{(\Hasp)'_\strong}$.
  
 {\it Sufficiency.} Let $g\in\Lsp{\infty}$. For an arbitrary $\funa\in\Hasp$ from \ref{ocenkalh} we have
 \begin{equation*}\int_{\Re^n}|g\Oplot\funa|\,d\lebd{n}\le \|g\|_{\Lsp{\infty}}\|\Oplot\funa\|_{\Lsp{1}}\le\frac{(2\pi)^{\frac{n}{2}}}{|\kappa_\varphi|}\|g\|_{\Lsp{\infty}}\|\funa\|_{\Hasp}.
 \end{equation*}
   \end{proof}

The following Lemma and Theorem characterize the ``weak'' associated space of $\Hasp$.

\begin{lemma} \label{kritandocenka} %Let $\varphi\in\Sshv$ and $\kappa_\varphi\not=0$.Then 
There is equivalence $g\in(\Hasp)'_\weak$ $\Leftrightarrow$  
\begin{equation}\label{ginYHasp}
g\in \left\{h\in\BMO{n}:\int_{\Re^n}|h\Oplot\funa|\,d\lebd{n}<\infty\ \forall\,\funa\in\Hasp\right\}.
\end{equation}
 And for $g\in (\Hasp)'_\weak$ there are the estimates
\begin{equation}\label{dvustoceka}
\frac{1}{4c^*_1(n,2,\varphi)}\|g\|_{\BMO{n}}\le \|g\|_{(\Hasp)'_\weak}\le c^*_2(n,\varphi)\|g\|_{\BMO{n}}.
\end{equation}
\end{lemma}

\begin{proof} Let \eqref{ginYHasp} be hold. Denote  $g^{(1)}:=\Real g$ and $g^{(2)}:=\Imag g$.  For $j\in\{1,2\}$, $k\in\Na$ and $\funa\in\Hasp$ we put 
\begin{equation*}
g_k^{(j)}(x):=\begin{cases}
      g^{(j)}(x), & |g^{(j)}(x)|<k,\\
      k, & g^{(j)}(x)>k,\\
      -k, & g^{(j)}(x)<-k; 
     \end{cases}\ \ \ \ \ 
\funb_k^{(j)}(\funa):=\int_{\Re^n}g_k^{(j)}\Oplot\funa\,d\lebd{n}. 
\end{equation*}
 Note that $g_k^{(j)}\in\BMO{n}$ and  $\|g_k^{(j)}\|_{\BMO{n}}\le 3\|g^{(j)}\|_{\BMO{n}}$.
Since $g_k^{(j)}\in\Lsp{\infty}$ then $\funb_k^{(j)}\in(\Hasp)^*$ and by Theorem  \ref{pofbmofunk}
\begin{equation*}
 |\funb_k^{(j)}(\funa)|\le c^*_2(n,\varphi)\|g_k^{(j)}\|_{\BMO{n}}\|\funa\|_{\Hasp}
\le 3c^*_2(n,\varphi)\|g^{(j)}\|_{\BMO{n}}\|\funa\|_{\Hasp}.
\end{equation*}
Further, $\int_{\Re^n}|g^{(j)}\Oplot\funa|\,d\lebd{n}\le\int_{\Re^n}|g\Oplot\funa|\,d\lebd{n}<\infty$ and by Lebesgue's dominated convergence theorem we obtain 
\begin{equation*}
 \left|\int_{\Re^n}g^{(j)}\Oplot\funa\,d\lebd{n}\right|\le  3c^*_2(n,\varphi)\|g^{(j)}\|_{\BMO{n}}\|\funa\|_{\Hasp}.
 \end{equation*}
 Hence,
\begin{equation*}
 \left|\int_{\Re^n}g\Oplot\funa\,d\lebd{n}\right|\le  6c^*_2(n,\varphi)\|g\|_{\BMO{n}}\|\funa\|_{\Hasp},
 \end{equation*}
 that is $g\in(\Hasp)'_\weak$. The functional $\funb:\Hasp\to\Co$, defined by formula
 \begin{equation*}
  \funb(\funa):=\int_{\Re^n}g\Oplot\funa\,d\lebd{n},\ \ \funa\in\Hasp,
 \end{equation*}
belongs $(\Hasp)^*$. Theorem \ref{sopryazh} implies the estimates \eqref{dvustoceka}.
\end{proof}

\begin{thm} %Let $\varphi\in\Sshv$ and $\kappa_\varphi\not=0$. Then
$(\Hasp)'_\weak=(\Lsp{\infty},\|\cdot\|_{\BMO{n}})$.
\end{thm}

\begin{proof} By Lemma \ref{kritandocenka} it is enough to prove that 
if  $f\in\Lsploc{1}$ and $\int_{\Re^n}|f\Oplot\funa|\,d\lebd{n}<\infty$ for all $\funa\in\Hasp$, then $f\in\Lsp{\infty}$.

Assume that  $f\not\in\Lsp{\infty}$. We denote by the symbol $E_f$ the set of all  Lebesgue points of the function $|f|$. Since $f\in\Lsploc{1}$ then $\lebd{n}(\Re^n\setminus E_f)=0$. The relation   $f\not\in\Lsp{\infty}$ implies the existence of a countable set of points $\{x_j\}_1^\infty\subset E_f$ such that $|f(x_k)|\ge k^2$ for $k\in\Na$.  The set $\{x_j\}_1^\infty$ either has a condensation point or is unbounded. In both cases, there is a subsequence $\{x_{j_k}\}_{k=1}^\infty$ and a set of cubes $\{I_k\}_{k=1}^\infty$ with the properties: $x_{ j_k}$ is the center of the cube $I_k$, $I_k\capt I_{k'}=\emptyset$ for $k\not= k'$, $\frac{1}{\lebd{n}(I_k)}\int_ {I_k}|f|\,d\lebd{n}\ge \frac{|f(x_{j_k})|}{2}$.  Note that  $|f(x_{j_k})|\ge k^2$ for $k\in\Na$.

For $k\in\Na$ let $a_{I_k}$ be the function constructed in \eqref{delimkub} for the cube $Q:=I_k$.  We put $h(x):=\sum_{k=1}^\infty\frac{1}{k^2} a_{I_k}(x)$, $x\in\Re^n$.
 Then $h\in\Lsp{1}$,   for any $k\in\Na$ the function $a_{I_k}$ is an atom, and for $\phi\in\Sshv$, $m\in\Na$ we have
 \begin{align*}
\left|\int_{\Re^n}h\phi\,d\lebd{n}-\int_{\Re^n}\left[\sum_{k=1}^m \frac{1}{k^2} a_{I_k}\right]\phi\,d\lebd{n}\right|&\le \|\phi\|_{\Lsp{\infty}}\sum_{k={m+1}}^\infty\frac{1}{k^2}\int_{I_k}|a_{I_k}|\,d\lebd{n}\\
&\le\|\phi\|_{\Lsp{\infty}}\sum_{k={m+1}}^\infty\frac{1}{k^2}.
 \end{align*}
By using Theorem \ref{LambdainH}, we obtain $\funf{h}\in\Hasp$.

Besides that,
\begin{equation*}
 \int_{\Re^n}|fh|\,d\lebd{n}\ge \sum_{k=1}^\infty\int_{I_k}|fh|\,d\lebd{n}=\sum_{k=1}^\infty\frac{1}{k^2\lebd{n}(I_k)}\int_{I_k}|f|\,d\lebd{n}\ge \frac{1}{2}\sum_{k=1}^\infty 1=\infty,
\end{equation*}
and we get a contradiction.
\end{proof}

Next theorem describes  the second associated spaces.

\begin{thm} %Let $\varphi\in\Sshv$ and $\kappa_\varphi\not=0$.
 \begin{listi}
  \item $g\in ((\Hasp)'_\weak)'_\strong$ $\Leftrightarrow$ $g=0$ $\lebd{n}$-a.e. on $\Re^n$.$\vphantom{\Big(}$
  \item $((\Hasp)'_\weak)'_\weak=\lOf{\Hasp}$.
 \end{listi}
\end{thm}

\begin{proof} \rom{1}. For $g\in ((\Hasp)'_\weak)'_\strong$ the equality $\int_{\Re^n}|gf|\,d\lebd{n}=0$ is necessary for any $f\in(\Hasp)'_\weak$ with $\|f\|_{(\Hasp)'_\weak}=0$.
Since $\chi_{\Re^n}\in (\Hasp)'_\weak$ and $\|\chi_{\Re^n}\|_{(\Hasp)'_\weak}=0$ then $g=0$ $\lebd{n}$-a.e. on $\Re^n$.

\rom{2}. Let $g\in ((\Hasp)'_\weak)'_\weak$. By definition for any $f\in (\Hasp)'_\weak$ the inequality  $\int_{\Re^n}|gf|\,d\lebd{n}<\infty$ holds and  
\begin{equation*}
\left|\int_{\Re^n}gf\,d\lebd{n}\right|\le \|g\|_{((\Hasp)'_\weak)'_\weak}\|f\|_{(\Hasp)'_\weak}. 
\end{equation*}
Since $\chi_{\Re^n}\in (\Hasp)'_\weak$ and  $\|\chi_{\Re^n}\|_{(\Hasp)'_\weak}=0$ then $g\in\Lsp{1}$ and $\int_{\Re^n}g\,d\lebd{n}=0$.

Let $M:=c^*_2(n,\varphi)\|g\|_{((\Hasp)'_\weak)'_\weak}$. For  
 $f\in\Cs{c}$ we put
$\Lambda f:=\int_{\Re^n}fg\,d\lebd{n}$. Then the estimate $|\Lambda f|\le M\|f\|_{\BMO{n}}$ holds and by the Hahn\,--\,Banach theorem \cite[3.3]{RudinFA} there exists a linear extension  $\tilde\Lambda$ of $\Lambda$ on $\VMO{n}$ with saving the estimate $|\tilde\Lambda f|\le M\|f\|_{\BMO{n}}$ for $f\in\VMO{n}$. By \cite[Theorem $(4.1)$]{CoifmanWeiss} there exists $\tilde g\in\lOf{\Hasp}$ such that
 $\tilde\Lambda\phi=\int_{\Re^n}\phi\tilde g\,d\lebd{n}$ for $\phi\in\Cs{c}$ and
\begin{equation*}
\|\tilde g\|_{\lOf{\Hasp}}=\|\tilde\Lambda\|_{(\VMO{n})^*}\le M.
\end{equation*}
Hence, $\tilde\Lambda\phi=\Lambda\phi$ for any $\phi\in\Cs{c}$, that is $\tilde g= g$ $\lebd{n}$-a.e. on $\Re^n$.
It implies  $g\in\lOf{\Hasp}$ and $\|g\|_{\lOf{\Hasp}}\le c^*_2(n,\varphi)\|g\|_{((\Hasp)'_\weak)'_\weak}$.

Conversely, $\lOf{\Hasp}\subset ((\Hasp)'_\weak)'_\weak$,
since for any $\funa\in\Hasp$ 
\begin{equation*}
\int_{\Re^n}|f\Oplot\funa|\,d\lebd{n}<\infty\ \ \forall\,f\in (\Hasp)'_\weak
\end{equation*}
and
\begin{equation*}
\left|\int_{\Re^n}f\Oplot\funa\,d\lebd{n}\right|\le \|\funa\|_{\Hasp}\|f\|_{(\Hasp)'_\weak}\ \ \forall\,f\in (\Hasp)'_\weak
\end{equation*}
hold. Besides that, $\|\Oplot\funa\|_{((\Hasp)'_\weak)'_\weak}\le \|\funa\|_{\Hasp}$.
\end{proof}

%\subsection{Associated spaces of the $\BMO{n}$}

%This section contains a description of the associated spaces of the $\BMO{n}$.
%$\phantom{q}$
Further we give a description of the associated spaces of the $\BMO{n}$.

\begin{thm} %Let $\varphi\in\Sshv$ and $\kappa_\varphi\not=0$.
\begin{listi}
  \item \label{bmoshtrixto} $(\BMO{n})'_\weak\subsetneq \lOf{\Hasp}$ and for  $g\in (\BMO{n})'_\weak$ the inequality 
  \begin{equation*}
\|g\|_{\lOf{\Hasp}}\le 4c^*_1(n,2,\varphi)c^*_2(n,\varphi)\|g\|_{(\BMO{n})'_\weak}   
  \end{equation*}
  holds
  \item \label{vidbmoslab}
$g\in(\BMO{n})'_\weak$  $\Leftrightarrow$  
\begin{equation}\label{ginYBMOweak}
g\in \left\{h\in\lOf{\Hasp}\,\middle|\,\int_{\Re^n}|hf|\,d\lebd{n}<\infty\ \forall\,f\in\BMO{n}\right\}.
\end{equation}
 And for $g\in (\BMO{n})'_\weak$ there are the estimates
\begin{equation}\label{dvustocekabmo}
\frac{1}{4c^*_1(n,2,\varphi)c^*_2(n,\varphi)}\|g\|_{\lOf{\Hasp}}\le \|g\|_{(\BMO{n})'_\weak}\le 6c^*_2(n,\varphi)\|g\|_{\lOf{\Hasp}}.
\end{equation}
   \item \label{tobmoshtrix} Let
\begin{equation*}
Y:=\left\{h\in \cupt_{p\in(1,\infty]}\Lsp{p}\,\middle|\, \supp h\text{ compact in }\Re^n,\ \int_{\Re^n} h\,d\lebd{n}=0\right\}.
\end{equation*}
Then $Y\subsetneq (\BMO{n})'_\weak$ and the closure of $Y$ in the space $\lOf{\Hasp}$ is  $\lOf{\Hasp}$.
 \end{listi}
 \end{thm}

\begin{proof} \rom{1}. For an arbitrary  $g\in (\BMO{n})'_\weak$ we have
\begin{align*}
 &\|g\|_{(\BMO{n})'_\weak}\ge \sup_{f\in \Lsp{\infty}:\|f\|_{\BMO{n}}\not=0}\frac{\left|\int_{\Re^n}g f\,d\lebd{n}\right|}{\|f\|_{\BMO{n}}}\\
 &\ge \sup_{f\in \Lsp{\infty}:\|f\|_{\BMO{n}}\not=0}\frac{\left|\int_{\Re^n}g f\,d\lebd{n}\right|}{4c^*_1(n,2,\varphi)\|f\|_{(\Hasp)'_\weak}}=\frac{1}{4c^*_1(n,2,\varphi)}\|g\|_{((\Hasp)'_\weak)'_\weak}\\
 &\ge \frac{1}{4c^*_1(n,2,\varphi)c^*_2(n,\varphi)}\|g\|_{\lOf{\Hasp}}.
\end{align*}
The example  \cite[IV, 6.2]{SteinHA} shows that $(\BMO{n})'_\weak\not=\lOf{\Hasp}$.

\rom{2}. Let \eqref{ginYBMOweak} be hold.  Fix an arbitrary $f\in\BMO{n}$. Approximating the function $f$ by functions from $\BMO{n}\capt\Lsp{\infty}$ as in the proof of Lemma \ref{kritandocenka}, we obtain the estimate 
\begin{equation*}
 \left|\int_{\Re^n}g f\,d\lebd{n}\right|\le  6c^*_2(n,\varphi)\|f\|_{\BMO{n}}\|g\|_{\lOf{\Hasp}},
 \end{equation*}
 that is $g\in (\BMO{n})'_\weak$ and $\|g\|_{(\BMO{n})'_\weak}\le 6c^*_2(n,\varphi)\|g\|_{\lOf{\Hasp}}$.
 
 \rom{3}. Let $p\in (1,\infty]$, $g\in\Lsp{p}$, $\supp g\subset Q$, for some cube $Q$, and $\int_{\Re^n} g\,d\lebd{n}=0$.  
 Fix an arbitrary $f\in\BMO{n}$. 
 Since (see \cite[IV, 1.3]{SteinHA}) $f\in \Lsploc{p'}$ then $\int_{\Re^n} |fg|\,d\lebd{n}<\infty$ and by \cite[IV, 1.3]{SteinHA}
\begin{align*}
 \left|\int_{\Re^n}fg\,d\lebd{n}\right|&=\left|\int_Q\Big(f-\Avg{f}{Q}\Big) g\,d\lebd{n}\right|\\
 &\le \left[\frac{1}{\lebd{n}(Q)}\int_Q\left|f-\Avg{f}{Q}\right|^{p'}\,d\lebd{n}\right]^\frac{1}{p'}\|g\|_{\Lsp{p}}\lebd{n}(Q)^\frac{1}{p'}\\
 &\le c(n,p)\|f\|_{\BMO{n}}\|g\|_{\Lsp{p}}\lebd{n}(Q)^\frac{1}{p'}. 
 \end{align*}
 It implies $g\in  (\BMO{n})'_\weak$.

 For  $k\in\Ze$ and $x\in\Re^n$ we put
 \begin{align*}
   &g(x):=\left[\chi_{\Re^n_{+}}(x)-\chi_{\Re^n\setminus\Re^n_{+}}(x)\right](1+|x|)^{-(n+1)},\ \ \lambda_k:=2^{n(k+2)}(1+2^k)^{-(n+1)},\\
   &a_k(x):=2^{-n(k+2)}(1+2^k)^{n+1}g(x)\chi_{\{y\in\Re^n:|y|\in [2^k,2^{k+1})\}}(x).
 \end{align*}
 Then $g\in\Lsp{1}$, each $a_k$ is an atom and $\sum_{k\in\Ze}|\lambda_k|<\infty$.
 Since $\sum_{j=k}^m \lambda_j \funf{a_j}\to \funf{g}$ as $m\to\infty$ and $k\to-\infty$   in weak\,$^*$ topology of $\Sshvf$, then $\funf{g}\in\Hasp$
by \ref{LambdainH}. Hence,  $g\in \lOf{\Hasp}$ and   $g\not\in Y$
   because $\supp g$ is not a compact in  $\Re^n$. However, by  \cite[IV, 1.1.4]{SteinHA} for any $f\in\BMO{n}$ the inequality $\int_{\Re^n} |fg|\,d\lebd{n}<\infty$ holds. And  \ref{vidbmoslab} implies $g\in (\BMO{n})'_\weak$.

 Since $\lOf{\Haspa{\infty}}\subset Y$ then the closure of $Y$ in the space $\lOf{\Hasp}$ is  $\lOf{\Hasp}$.
\end{proof}

 \begin{corollary}
  The space $(\BMO{n})'_\weak$ is not complete.
 \end{corollary}

 \begin{proof}
  Suppose that $(\BMO{n})'_\weak$ is complete. Fix an arbitrary $g\in \lOf{\Hasp}$. By Theorem \ref{sushrazlozh} there exists $\{g_k\}_1^\infty\subset \lOf{\Haspa{\infty}}$ with property $\|g_k-g\|_{\lOf{\Hasp}}\to 0$ as $k\to\infty$. Then  $\{g_k\}_1^\infty$ is a Cauchy sequence in $\lOf{\Hasp}$. By \ref{tobmoshtrix} $\{g_k\}_1^\infty$ is a Cauchy sequence in $(\BMO{n})'_\weak$ too. By assumption there exists $g_0\in (\BMO{n})'_\weak$ such that $\|g_k-g_0\|_{(\BMO{n})'_\weak}\to 0$ as $k\to\infty$. Applying \ref{bmoshtrixto}, we obtain  $\|g_k-g_0\|_{\lOf{\Hasp}}\to 0$ as $k\to\infty$, that is $g=g_0\in (\BMO{n})'_\weak$. We got a contradiction. 
 \end{proof}

 \begin{thm} %Let $\varphi\in\Sshv$ and $\kappa_\varphi\not=0$.
 \begin{listi}
   \item $g\in (\BMO{n})'_\strong$ $\Leftrightarrow$ $g=0$ $\lebd{n}$-a.e. on $\Re^n$.$\vphantom{\Big(}$
  \item $((\BMO{n})'_\weak)'_\weak=\BMO{n}$.
  \item $((\BMO{n})'_\weak)'_\strong=\Lsp{\infty}$.$\vphantom{\Big(}$
 \end{listi}  
 \end{thm}

 \begin{proof}   \rom{1}. For  $g\in (\BMO{n})'_\strong$ the equality $\int_{\Re^n}|gf|\,d\lebd{n}=0$ is necessary for any $f\in \BMO{n}$ with $\|f\|_{\BMO{n}}=0$. In case of $f=\chi_{\Re^n}$ we have $g=0$ $\lebd{n}$-a.e. on $\Re^n$.
  
  \rom{2}. Fix an arbitrary $g\in ((\BMO{n})'_\weak)'_\weak$. Let $M:=6c^*_2(n,\varphi)\|g\|_{((\BMO{n})'_\weak)'_\weak}$. For any $f\in \lOf{\Haspa{\infty}}\subset (\BMO{n})'_\weak$ we define 
  $\Lambda f:=\int_{\Re^n}fg\,d\lebd{n}$. Note that
  \begin{equation*}
|\Lambda f|\le \|g\|_{((\BMO{n})'_\weak)'_\weak}\|f\|_{(\BMO{n})'_\weak}\le  M\|f\|_{\lOf{\Hasp}}.
 \end{equation*}
 By the Hahn\,--\,Banach theorem \cite[3.3]{RudinFA} there exists a linear extension  $\tilde\Lambda$ of $\Lambda$ on $\lOf{\Hasp}$ with saving the estimate  $|\tilde\Lambda f|\le M\|f\|_{\lOf{\Hasp}}$, $f\in\lOf{\Hasp}$. By \ref{predstfunkhardysp} there exists $\tilde g\in\BMO{n}$ such that
$\tilde\Lambda f=\int_{\Re^n}f\tilde g\,d\lebd{n}$ for $f\in\lOf{\Haspa{\infty}}$ and the estimates
\begin{equation*}
\|\tilde g\|_{\BMO{n}}\le 4c^*_1(n,2,\varphi)\|\tilde\Lambda\|_{(\lOf{\Hasp})^*}\le 4c^*_1(n,2,\varphi)M.
\end{equation*}
hold. Then $\tilde\Lambda f=\Lambda f$ for $f\in\lOf{\Haspa{\infty}}$. Since $\Big\{D^if\,|\,f\in \CDs{c}{1},\,i\in\{1,\ldots,n\}\Big\}\subset\lOf{\Haspa{\infty}}$ there exists \cite[1.1.11]{Mazya} (all weak first order derivatives of the function $\tilde g- g$ are equal to zero) a constant $\lambda\in\Co$ such that $\tilde g-g=\lambda$ $\lebd{n}$-a.e. on $\Re^n$.
Hence,   $g\in\BMO{n}$ and 
\begin{equation*}
\|g\|_{\BMO{n}}\le 24c^*_2(n,\varphi)c^*_1(n,2,\varphi)\|g\|_{((\BMO{n})'_\weak)'_\weak}. 
\end{equation*}

Conversely, fix an arbitrary $g\in\BMO{n}$. For any $f\in (\BMO{n})'_\weak$  \ref{vidbmoslab} implies $\int_{\Re^n}|fg|\,d\lebd{n}<\infty$ and  
\begin{equation*}
 \left|\int_{\Re^n}fg\,d\lebd{n}\right|\le \|g\|_{\BMO{n}}\|f\|_{(\BMO{n})'_\weak},
\end{equation*}
that is $g\in ((\BMO{n})'_\weak)'_\weak$ and $\|g\|_{((\BMO{n})'_\weak)'_\weak}\le \|g\|_{\BMO{n}}$.

\rom{3}. Fix an arbitrary $g\in ((\BMO{n})'_\weak)'_\strong$. Then, similarly to the proof \ref{silnohardiacco} we have
\begin{equation*}
 \|g\|_{((\BMO{n})'_\weak)'_\strong}\ge \sup_{f\in \Haspa{\infty}\setminus\{0\}}\frac{\int_{\Re^n}|g f|\,d\lebd{n}}{6c^*_2(n,\varphi)\|f\|_{\lOf{\Hasp}}}=\frac{1}{6c^*_2(n,\varphi)}\|g\|_{\Lsp{\infty}}.
\end{equation*}

Conversely, fix an arbitrary  $g\in\Lsp{\infty}$. For any $f\in (\BMO{n})'_\weak\subset\lOf{\Hasp}\subset\Lsp{1}$   there are the estimates
\begin{equation*}
 \int_{\Re^n}|fg|\,d\lebd{n}\le \|g\|_{\Lsp{\infty}}\|f\|_{\Lsp{1}}\le \|g\|_{\Lsp{\infty}}\frac{(2\pi)^{\frac{n}{2}}4c^*_1(n,2,\varphi)c^*_2(n,\varphi)}{|\kappa_\varphi|}\|f\|_{(\BMO{n})'_\weak},
\end{equation*}
that is $g\in ((\BMO{n})'_\weak)'_\strong$.
 \end{proof}

\end{document}